\documentclass [12pt]{amsart}

\usepackage{graphicx}
\usepackage{amsmath}
\usepackage{amssymb}
\usepackage{enumerate}

\newtheorem{thm}{Theorem}[section]
\newtheorem{Theorem}[thm]{Theorem}

\newtheorem{lemma}{Lemma}[section]

\newtheorem{corol}{Corollary}[section]

\newtheorem{prop}{Proposition}[section]

\newtheorem{rmk}{Remark}[section]

\numberwithin{equation}{section}

\newcommand{\af}{\mathfrak{a}}

\newcommand{\bfrac}[2]{\left(\frac{#1}{#2}\right)}
\newcommand{\tRe}{\textup{Re}}

\begin{document}

\title{Conditional bounds for the least quadratic non-residue and related problems}

\author{Youness Lamzouri}
\address{Department of Mathematics and Statistics, York University,
4700 Keele Street,
Toronto, ON, M3J1P3}
\email{lamzouri@mathstat.yorku.ca}

\author{Xiannan Li}
\address{Department of Mathematics, University of Illinois at Urbana-Champaign,
1409 W. Green Street,
Urbana, IL, 61801}
\email{xiannan@illinois.edu}

\author{Kannan Soundararajan}
\address{Department of Mathematics, Stanford University, Stanford, CA 94305}
\email{ksound@math.stanford.edu}


\date{\today}
\begin{abstract} This paper studies explicit and theoretical bounds for several interesting quantities in number theory, conditionally on the Generalized Riemann Hypothesis.  Specifically, we improve the existing explicit bounds for the least quadratic non-residue and the least prime in an arithmetic progression. We also refine the classical conditional bounds of Littlewood for $L$-functions at $s=1$. In particular, we derive explicit upper and lower bounds for $L(1,\chi)$ and $\zeta(1+it)$, and deduce explicit bounds for the class number of imaginary quadratic fields.  Finally, we improve the best known theoretical bounds for the least quadratic non-residue, and more generally, the least $k$-th power non-residue.
\end{abstract}

\thanks{The first author is supported in part by an NSERC Discovery grant. The third author is supported in part by NSF grant DMS-1001068, and a Simons Investigator grant from the Simons Foundation}

\maketitle

\section{Introduction}

\noindent  Let $q>1$ be a natural number and let $G=({\Bbb Z}/q{\Bbb Z})^*$ denote the group 
of reduced residues $\pmod q$.  Given a proper subgroup $H$ of $G$,  
two natural and interesting questions in number theory are: 
\begin{enumerate}
\item Determine or estimate the 
least prime $p$ not dividing $q$ and lying outside the subgroup $H$, and
\item Given a 
coset of $H$ in $G$, determine or estimate the least prime $p$ lying in that coset.  
\end{enumerate}
These problems have a long history, and have attracted the attention of many 
mathematicians.  There are two particular cases that are especially well known.  If $q$ is 
prime and $H$ is the group of quadratic residues the first problem amounts to the 
famous question of Vinogradov on the least quadratic non-residue.  For any 
$q$, if  $H=\{1\}$ is the trivial subgroup, then the second problem is that 
of estimating the least prime in an arithmetic progression.   In this paper we assume the 
truth of the Generalized Riemann Hypothesis and obtain explicit as well as asymptotic bounds 
on these problems.  Our work improves upon previous results in this 
area, notably the works of Bach \cite{Bach1}, and Bach and Sorenson \cite{Bach2}.

\subsection{The least prime outside a subgroup} Assuming GRH, Ankeny 
\cite{Ankeny} established that the least prime lying outside a proper 
subgroup $H$ of $({\Bbb Z}/q{\Bbb Z})^*$ is $O((\log q)^2)$.  Ankeny's work was refined by Bach \cite{Bach1}, 
who gave explicit as well as asymptotic 
bounds for the least such prime.  Thus Bach showed, on GRH, that the least prime $p$ 
that does not lie in $H$ is less than $2 (\log q)^2$, and if  the 
additional natural requirement that $p \nmid q$ is imposed then $p$ is less than $3 (\log q)^2$.  
He further established that the least prime $p\nmid q$ with $p$ not lying 
in $H$ satisfies the asymptotic bound $p\le (1+o(1)) (\log q)^2$.   In the 
special case when $H$ consists of the square residue classes $\pmod q$, 
Wedeniwski \cite{Wed} has recently given the improved explicit bound 
$\frac{3}{2}(\log q)^2 -  \frac{44}{5} \log q + 13$.   
We begin with an explicit 
result on this question.  We shall assume that $q\ge 3000$, since questions 
for smaller values of $q$ may easily be settled by direct calculation. As usual
we denote by $\omega(q)$ the number of distinct prime factors of $q$. 

 \begin{Theorem}\label{thmexplicit}
Assume GRH. Let $q\geq 3000$ be an integer, and let $H$ be a proper 
subgroup of $G=({\Bbb Z}/q{\Bbb Z})^*$.
\begin{enumerate}
\item Let $A(q) = \max(0, 2\log \log q -8/5 - \sum_{p|q} (\log p)/(p-1))$, and 
put $$B(q) = \max(0, 2\log \log q + 3 + 2\omega(q)(\log \log q)^2/\log q - 2A(q)).$$    The least prime $\ell$ with $\ell \nmid q$ and $\ell$ not lying 
in the subgroup $H$ satisfies the bound 
$\ell \le (\log q+ B(q))^2$.  
\item Suppose $q$ is not divisible by any prime below $(\log q)^2$.  
Then there is a prime $\ell \le (\log q)^2$ with $\ell$ not lying in the 
subgroup $H$.  
\end{enumerate}
\end{Theorem}
 
The quantity $B(q)$ in  Theorem \ref{thmexplicit} is zero when $q$ is large and 
does not have too many small prime factors.  Asymptotically $B(q)=o(\log \log q)$, and in any given situation it may easily be directly 
estimated.  
 
Consider now  the special case when $q$ is 
prime and $H$ is the subgroup of quadratic residues $\pmod q$.  
 Combining Theorem \ref{thmexplicit} with a computer calculation for the primes below $3000$ we 
obtain the following corollary. 

\begin{corol} 
\label{corexplicit}  Assume GRH.  If $q\ge 5$ is prime, the least quadratic non-residue $\pmod q$ 
lies below $(\log q)^2$. 
\end{corol}
 
Theorem \ref{thmexplicit} establishes an explicit bound of the same strength as 
the asymptotic bound given in Bach \cite{Bach1}.   Interestingly it turns out that the 
asymptotic bound in this question may also be improved.

\begin{Theorem}\label{thmtheoretical1}
Assume GRH.  Let $q$ be a large integer and let $H$ be a subgroup 
of $G=({\Bbb Z}/q{\Bbb Z})^*$ with index $h=[G:H]>1$.  Then the least prime $p$ not in $H$ satisfies 
$$p < (\alpha(h)+o(1))(\log q)^2,$$
where   $\alpha(2) = 0.8$, $\alpha(3) = 0.7$ and in general $\alpha(h) = 0.66$ for all $h>3$.  
\end{Theorem}
  
In particular the least quadratic non-residue is bounded by $(0.8+o(1))(\log q)^2$.   Theorem \ref{thmtheoretical1}  gives good bounds when $h$ is small, but when $h$ is very large the following stronger result holds, which appears 
to be the limit of our method.  

\begin{Theorem}\label{thmtheoretical2}
Assume GRH.  Let $q$ be a large integer and let $H$ be a subgroup 
of $G=({\Bbb Z}/q{\Bbb Z})^*$ with index $h=[G:H]\ge 4$.  Then the least prime $p$ not in $H$ satisfies 
$$
p < \Big(\frac 14+o(1)\Big) \Big(1-\frac 1h\Big)^2 \Big(\frac{\log (2h)}{\log (2h)-4}\Big)^2 (\log q)^2.
$$
\end{Theorem}

\subsection{The least prime in a given coset}  Now we turn to the second of 
our questions: the problem of determining the least prime lying in a given coset 
of $H$.

\begin{Theorem}\label{thmcoset}
Assume GRH.  Let $q\geq 20000$ and let $H$ be a subgroup of $G=({\Bbb Z}/q{\Bbb Z})^*$ with index $h=[G:H]>1$.  Let $p$ be the smallest prime lying in a given coset $aH$.  Then either $p \leq 10^9$, or
$$p\leq \Big((h-1)\log q + 3(h+1) + \frac 52(\log\log q)^2\Big)^2.$$  
\end{Theorem}

From the proof one may obtain more precise estimates, but the 
form above seems easiest to state.  

We next single out the case when $H$ is the trivial subgroup comprising of just the identity.  Here the problem amounts to getting bounds on the least prime $P(a,q)$ in the
arithmetic progression $a\pmod q$ (where we assume that $(a,q)=1$).  
A celebrated (unconditional) theorem of  Linnik  gives that $P(a,q)\ll q^{L}$ 
for some absolute constant $L$.  The work of  Heath-Brown \cite{HB} shows that $L=5.5$ is an admissible value for Linnik's constant, and this has been recently improved to $L=5.2$ by Xylouris \cite{Xy}.

Under the assumption of the Generalized Riemann Hypothesis, Bach and Sorenson \cite{Bach2} showed that
\begin{equation}\label{BachSorenson1}
P(a,q)\leq (1+o(1))(\phi(q)\log q)^2,
\end{equation}
and they also derived the explicit bound
\begin{equation}\label{BachSorenson2}
P(a,q)\leq 2(q\log q)^2.
\end{equation}
Our work leading up to Theorem \ref{thmcoset} and some computation for $q\leq 20000$ permits the following refinement of these bounds.  
 
\begin{corol}\label{thmexplicit2}
Assume GRH. Let $q>3$, and let $a\pmod q$ be a reduced residue class.  
The least prime in the arithmetic progression $a\pmod q$ satisfies 
$$
P(a,q)\leq (\phi(q)\log q)^2.
$$
\end{corol}
 
By modifying the argument for Theorems \ref{thmtheoretical1} and \ref{thmtheoretical2}, 
and using the Brun-Titchmarsh theorem, one may derive an asymptotic 
bound of the form $P(a,q) \le (1-\delta+o(1)) (\phi(q)\log q)^2$ for 
some small $\delta >0$.

\subsection{Bounds for values of $L$-functions at $s=1$}

The Generalized Riemann Hypothesis implies the Generalized Lindel{\" o}f Hypothesis, 
thus furnishing good upper bounds for the values of $L$-functions.  Such 
results go back to Littlewood \cite{Li1}, \cite{Li2}, and the best known asymptotic bounds 
may be found in \cite{ChSo} (for $L$-values on the critical line) and \cite{CaCh} 
(for $L$-values inside the critical strip).  Explicit bounds on GRH  for $L$-values on the critical 
line may be found in \cite{Ch}.    The methods of this paper allow us to 
give, assuming GRH, new explicit upper and lower bounds for $L$-values at the 
edge of the critical strip.  

\begin{thm}\label{L1chi}
Assume GRH. Let $q$ be a positive integer and $\chi$ be a primitive character modulo $q$.  For 
$q \ge 10^{10}$ we have
$$
|L(1,\chi)|\leq 2e^{\gamma}\left(\log\log q-\log 2+\frac{1}{2}+ \frac{1}{\log\log q}\right),
$$
and
$$\frac{1}{|L(1,\chi)|}\leq \frac{12e^{\gamma}}{\pi^2} \Big( \log \log q -\log 2 +\frac 12 + \frac{1}{\log \log q}+\frac{14\log\log q}{\log q}\Big).$$
\end{thm}


By Dirichlet's class number formula, we can deduce explicit bounds for the class number of an imaginary quadratic field under the assumption of GRH. More specifically, let $-q$ be a fundamental discriminant with $q>4$ and $\chi_{-q}(n)=\left(\frac{-q}{n}\right)$ be the Kronecker symbol, which is a primitive Dirichlet character $\pmod q$. Moreover, denote by $h(-q)$ the class number of the imaginary quadratic field $\mathbb{Q}(\sqrt{-q})$. Then, the Dirichlet class number formula reads
\begin{equation}\label{ClassFormula}
 h(-q)= \frac{\sqrt{q}}{\pi}L(1,\chi_{-q}).
\end{equation}
For discussions on the computation of these 
class numbers, we refer to \cite{Bo}, \cite{HaMc}, \cite{JRW}, \cite{Wat}.  From Theorem \ref{L1chi} 
we obtain the following corollary.  

\begin{corol}\label{ClassBounds} Assume GRH. Let $-q$ be a fundamental discriminant, and 
let $h(-q)$ denote the class number of ${\Bbb Q}(\sqrt{-q})$.  If $q\ge 10^{10}$ 
then 
$$
h(-q)\geq \frac{\pi}{12e^{\gamma}} \sqrt{q}\Big( \log \log q -\log 2 +\frac 12 + \frac{1}{\log \log q}+\frac{14\log\log q}{\log q}\Big)^{-1},
$$ 
and
$$
h(-q)\leq \frac{2e^{\gamma}}{\pi}\sqrt{q}\left(\log\log q-\log 2+\frac{1}{2}+ \frac{1}{\log\log q}\right).
$$
\end{corol}

The work of  \cite{JRW} computes class numbers and class groups for all imaginary quadratic 
fields with absolute discriminant below $10^{11}$.  From our corollary it follows that on GRH $h(-q) \ge 9053$ 
when $q\ge 10^{11}$.  

Our methods would also lead to explicit bounds for more general $L$-functions.  
We do not pursue this here, and content ourselves by stating that for real 
numbers $t \ge 10^{10}$ we have (assuming RH) 
$$
|\zeta(1+it)|\leq 2e^{\gamma}\left(\log\log t-\log 2+\frac{1}{2}+ \frac{1}{\log\log t}\right),
$$
and
$$
\frac{1}{|\zeta(1+it)|}\leq \frac{12e^{\gamma}}{\pi^2} \Big( \log \log t -\log 2 +\frac 12 + \frac{1}{\log \log t}+\frac{14\log\log t}{\log t}\Big).
$$


\section{Preliminary lemmas}

\noindent This section collects together several preliminary results that will be useful in 
our subsequent work.   We first introduce a convenient notation that will be in place 
for the rest of the paper.  We let $\theta$ stand for a complex number of magnitude at 
most one.  In each occurrence $\theta$ might stand for a different value, so that we may write 
$\theta -\theta =2\theta$, $\theta \times \theta =\theta$ and so on. 

Recall from Chapter 12 of \cite{Da} the following properties of the Riemann zeta-function and 
Dirichlet $L$-functions.   Set
$$
\xi(s) = s(s-1) \pi^{-\frac s2} \Gamma\Big(\frac s2\Big) \zeta(s),
$$
which is an entire function of order $1$, satisfies the functional equation $\xi(s) = \xi(1-s)$, 
and for which the
Hadamard factorization formula gives
\begin{equation}
\label{Had3}
\xi(s) = e^{Bs} \prod_{\rho} \Big(1-\frac s\rho\Big) e^{s/\rho}.
\end{equation}
Above $\rho$ runs over the non-trivial zeros of $\zeta(s)$, and $B$ is a
real number given by
\begin{equation}
\label{Had4}
B= -\sum_{\rho} \tRe \frac{1}{\rho}  =  \frac 12\log (4\pi) - 1 -\frac{\gamma}{2}
= - 0.02309\ldots
\end{equation}
 
 Let $\chi \pmod q$ denote a primitive Dirichlet character, and let $L(s,\chi)$ denote the
associated Dirichlet $L$-function.  Put ${\frak a}=(1-\chi(-1))/2= 0$ or $1$ depending on whether
the character $\chi$ is even or odd.  Let
$\xi(s,\chi)$ denote the completed $L$-function
$$
\xi(s,\chi) = \Big(\frac q{\pi}\Big)^{\frac s2} \Gamma \Big(\frac{s+{\frak a}}{2}\Big) L(s,\chi).
$$
The completed $L$-function satisfies the functional equation
$$
\xi(s,\chi)=\epsilon_\chi \xi(1-s,\overline{\chi})
$$
where $\epsilon_\chi$ is a complex number of size $1$.   The zeros of $\xi(s,\chi)$ are precisely the 
non-trivial zeros of $L(s,\chi)$, and letting $\rho_\chi = \frac 12+ i\gamma_\chi$ denote such a  zero 
(throughout the paper we assume the truth of the GRH), we have Hadamard's factorization
formula:
\begin{equation}
\label{Had}
\xi(s,\chi) = \exp(A(\chi) + s B(\chi)) \prod_{\rho_\chi} \Big(1-\frac{s}{\rho_\chi}\Big) e^{s/\rho_\chi}.
\end{equation}
Above, $A(\chi)$ and $B(\chi)$ are constants, with $\tRe B(\chi)$ being of particular
interest for us.  Note that
\begin{equation}
\label{Had2}
\tRe (B(\chi))  = \tRe (B(\overline{\chi}))
= \tRe \frac{\xi^{\prime}}{\xi}(0,\chi)=-\sum_{\rho_\chi} \tRe \frac{1}{\rho_\chi}.
\end{equation}

Recall the digamma function $\psi_0(z) =\frac{\Gamma^{\prime}}{\Gamma}(z)$, and 
its derivative the trigamma function $\psi_1(z) = \psi_0^{\prime}(z)$.
The following special values will be useful: 
$$
\psi_0(1)= -\gamma, \ \ \psi_1(1) =
\zeta(2) =\frac{\pi^2}6, \ \  \psi_0(1/2)= -2\log 2 -\gamma, \ 
\text{and} \ 
\psi_1(1/2)=\frac{\pi^2}2.
$$
\begin{lemma}
\label{lem1}
Assume RH.  For $x>1 $ there exists some $|\theta|\le 1$ such that 
$$
\sum_{n\le x} \Lambda(n) \log (x/n) =x - (\log 2\pi) \log x  -1 + \sum_{k=1}^{\infty}\frac{1-x^{-2k}}{4k^2} +2\theta |B| (\sqrt{x}+1).
$$
\end{lemma}

\begin{lemma}
\label{lem2}
Assume GRH.  Let $q\ge 3$ and let $\chi \pmod q$ be a primitive Dirichlet character.  
Define 
$$ 
S(x,\chi):= \sum_{n\le x} \Lambda(n) \chi(n) \log (x/n).
$$
For $x>1$, there exists some $|\theta|\le 1$ such that
$$
S(x,\chi) = |\textup{Re}(B(\chi))| (2\theta\sqrt{x}
+2\theta) - \frac{\xi '}{\xi}(0, \chi) \log x
+ \frac 12 \Big(\log \frac q{\pi}\Big) \log x + {\widetilde E}_{\frak a}(x),
$$
where
$$
{\widetilde E}_0(x) = \frac{\pi^2}{24}-\frac{\gamma}{2}\log x -\frac 12(\log x)^2  -\sum_{k=1}^{\infty}
\frac{x^{-2k}}{(2k)^2},
$$
and
$$
{\widetilde E}_1(x) =\frac{\pi^2}{8}- (\log 2+\gamma/2) \log x -\sum_{k=0}^{\infty} \frac{x^{-2k-1}}{(2k+1)^2} .
$$
In particular, 
$$
\tRe S(x,\chi) = |\textup{Re}(B(\chi))| (2\theta\sqrt{x}
+2\theta +\log x)
+ \frac 12 \Big(\log \frac q{\pi}\Big) \log x + {\widetilde E}_{\frak a}(x).
$$
\end{lemma}

We shall confine ourselves to proving Lemma \ref{lem2}; the proof of 
Lemma \ref{lem1} is essentially the same.  

\begin{proof}[Proof of Lemma \ref{lem2}] We begin with
$$
\frac{1}{2\pi i} \int_{2-i\infty}^{2+i\infty} -\frac{\xi^{\prime}}{\xi}(s,\chi) \frac{x^s}{s^2} ds,
$$
and first evaluate the integral by moving the line of integration to the left.
There are poles at the non-trivial zeros $\rho_\chi$ of $L(s,\chi)$, and a double pole at $s=0$.
Evaluating the residues here we find that our integral equals
$$
-\sum_{\rho_\chi} \frac{x^{\rho_\chi}}{\rho_\chi^2} -  \frac{\xi^{\prime}}{\xi}(0,\chi)
\log x - \Big(\frac{\xi^{\prime}}{\xi}\Big)^{\prime}(0,\chi).
$$
Invoking GRH,
 \begin{equation*}
-\sum_{\rho_\chi} \frac{x^{\rho_\chi}}{\rho_\chi^2} = \theta \sqrt{x} \sum_{\rho_\chi}
\frac{1}{|\rho_\chi|^2} =2\theta \sqrt{x} |\text{Re} (B(\chi))|.
\end{equation*}
Further
$$
\Big(\frac{\xi^{\prime}}{\xi}\Big)^{\prime}(0,\chi) = -\sum_{\rho_\chi} \frac{1}{\rho_\chi^2} =
2\theta |\text{Re}(B(\chi))|.
$$
Putting these observations together we obtain 
\begin{equation}
\label{lem2.1}
\frac{1}{2\pi i} \int_{2-i\infty}^{2+i\infty} -\frac{\xi^{\prime}}{\xi}(s,\chi) \frac{x^s}{s^2} ds
= |\text{Re}(B(\chi))| \Big(2\theta \sqrt{x}+2\theta\Big) - \frac{\xi '}{\xi}(0, \chi) \log x.
\end{equation}

On the other hand, the LHS of \eqref{lem2.1} equals
\begin{equation}
\label{lem2.2}
\frac{1}{2\pi i} \int_{2-i\infty}^{2+i\infty}
\Big(-\frac 12 \log \frac{q}{\pi} -\frac{1}{2}\psi_0\Big(\frac{s+\frak a}{2}\Big)-\frac{L^{\prime}}{L}(s,\chi) \Big)\frac{x^s}{s^2} ds.
\end{equation}
The first and third terms above contribute
$$
-\frac 12 \Big(\log \frac q\pi \Big) \log x + \sum_{n\le x} \Lambda(n) \chi(n)
\log (x/n).
$$
If $\frak a =1$, the middle term in \eqref{lem2.2} equals (here there are simple
poles at the negative odd integers, and a double pole at $0$)
$$
\sum_{k=0}^{\infty}\frac{x^{-2k-1}}{(2k+1)^2} - \frac 12 \psi_0\Big(\frac 12\Big) \log x
- \frac 14 \psi_1\Big(\frac 12\Big) = - {\widetilde E}_1(x).
$$
If $\frak a=0$ the middle term in \eqref{lem2.2} equals (here there are
simple poles at the negative even integers, and a triple pole at $0$)
$$
\sum_{k=1}^{\infty} \frac{x^{-2k}}{(2k)^2} + \frac 12 (\log x)^2
-\frac 12\psi_0(1) \log x -\frac 14 \psi_1(1)= - {\widetilde E}_0(x).
$$
This completes our proof of the first assertion.  The second assertion follows by taking real part, and noting that $\text{Re }\frac{\xi^{\prime}}{\xi}(0,\chi) = \text{Re}(B(\chi))$ from \eqref{Had2}.
\end{proof}

\begin{lemma}
\label{lem3}  Assume GRH.  
Let $q\ge 3$ and let $\chi \pmod q$ be a primitive Dirichlet character. 
For any $x>1 $ we have, for some $|\theta|\le 1$,
\begin{align*}
 -\frac{\xi '}{\xi}(0, \bar{\chi}) -& \frac{1}{x} \frac{\xi '}{\xi}(0,\chi) + \frac{2\theta}{\sqrt{x}} |\textup{Re}(B(\chi))| 
\\
&=  \frac 12\Big(1-\frac 1x\Big)\log \frac q{\pi} - \sum_{n\le x}\frac{\Lambda(n)}{n}\chi(n)
\Big(1-\frac nx\Big) + E_{\frak a}(x),
\end{align*}
where
$$
E_0(x)=-\log 2 -\frac{\gamma}{2}\Big(1-\frac 1x\Big)+ \frac{\log x+1}{x}-
\sum_{k=1}^{\infty} \frac{x^{-2k-1}}{2k(2k+1)} ,
$$
and
$$
E_1(x) = -\sum_{k=0}^{\infty} \frac{x^{-2k-2}}{(2k+1)(2k+2)} -\frac{\gamma}{2}\Big(1-\frac 1x\Big) + \frac{\log 2}{x}.
$$
In particular, $|\textup{Re}(B(\chi))|$ equals 
\begin{align*}
  \Big(1+\frac{2\theta}{\sqrt{x}}+\frac{1}{x}\Big)^{-1}
\Big(\frac 12\Big(1-\frac 1x\Big)\log \frac q{\pi} - \textup{Re}\sum_{n\le x}\frac{\Lambda(n)}{n}\chi(n)
\Big(1-\frac nx\Big) + E_{\frak a}(x)\Big).
\end{align*}
\end{lemma}
\begin{proof}  The proof is similar to that of Lemma \ref{lem2}, and so we will be 
brief. We  consider
\begin{equation}
\label{lem3integral}
\frac{1}{2\pi i} \int_{2-i\infty}^{2+i\infty} \frac{\xi^{\prime}}{\xi}(s,\chi) \frac{x^{s-1}}{s(s-1)} ds,
\end{equation}
and begin by evaluating the integral by moving the line of integration to the left.  There are
poles at $s=1$, $s=\rho_\chi$ and $s=0$ and therefore our integral above equals
$$
 \frac{\xi^{\prime}}{\xi}(1,\chi) -\frac{1}{x} \frac{\xi^{\prime}}{\xi}(0,\chi) +
\sum_{\rho_\chi} \frac{x^{\rho_\chi-1}}{\rho_\chi(\rho_\chi-1)}.
$$
Note that $\frac{\xi^{\prime}}{\xi}(1,\chi) = - \frac{\xi^{\prime}}{\xi}(0, \bar{\chi})$.  Further, since we are assuming GRH, we have for some $|\theta|\le 1$,
$$
\sum_{\rho_\chi} \frac{x^{\rho_\chi-1}}{\rho_\chi (\rho_\chi-1)} =\frac{ \theta }{\sqrt{x}} \sum_{\rho_\chi}
\frac{1}{|\rho_\chi|^2}= \frac{2\theta}{\sqrt{x}} |\text{Re}(B(\chi))|.
$$
Thus the quantity in \eqref{lem3integral} equals the LHS of the first identity claimed in 
our lemma.  

On the other hand, we may rewrite the integral \eqref{lem3integral} as
\begin{align}
\label{lem3.2}
\frac{1}{2\pi i} \int_{2-i\infty}^{2+i\infty}
\Big( \frac 12 \log \frac{q}{\pi} + \frac 12\psi_0\Big(\frac{s+\frak a}{2}\Big)
+ \frac{L^{\prime}}{L}(s,\chi)\Big) \frac{x^{s-1}}{s(s-1)} ds.
\end{align}
Now the first and the last terms above give
$$
\frac 12 \Big(1-\frac 1x\Big) \log \frac{q}{\pi}
- \sum_{n\le x} \frac{\Lambda(n)}{n} \chi(n) \Big(1-\frac nx\Big).
$$
If $\frak a=1$ then the middle term in \eqref{lem3.2} equals (there are simple
poles at $s=1$, $0$, and the negative odd integers)
$$
-\sum_{k=0}^{\infty} \frac{x^{-2k-2}}{(2k+1)(2k+2)} + \frac 12\psi_0(1)-
\frac{1}{2x} \psi_0\Big(\frac12\Big)
=E_1(x).
 $$
 If $\frak a=0$ then the middle term equals (here there are simple
 poles at $s=1$ and at the negative even integers, and a double pole at $s=0$)
 $$
 \frac{1}{2} \psi_0\Big(\frac 12\Big) +
 \frac{\log x+1+ \gamma/2}{x} - \sum_{k=1}^{\infty} \frac{x^{-2k-1}}{2k(2k+1)} = E_0(x).
 $$
This proves the first part of our lemma.

Now, by the functional equation and \eqref{Had2}, we have
$\text{Re} \frac{\xi^{\prime}}{\xi}(1,\chi) =- \text{Re} \frac{\xi^{\prime}}{\xi}(0,\overline{\chi})=
-\text{Re}(B(\chi))$.  Thus taking real parts in the identity just established, 
 we deduce the stated identity for $|\text{Re}(B(\chi))|$.
\end{proof}

\begin{lemma} \label{lem4} Assume RH.  For $x>1$ we have, for some $|\theta| \le 1$, 
$$ 
\sum_{n\le x} \frac{\Lambda(n)}{n} \Big( 1-\frac nx\Big) =  \log x - (1+\gamma) + 
\frac{\log (2\pi)}{x} - \sum_{n=1}^{\infty} \frac{x^{-2n-1}}{2n(2n+1)} 
+ 2 \theta \frac{|B|}{\sqrt{x}}. 
$$ 
\end{lemma} 
\begin{proof}  We argue as in Lemma \ref{lem3}, starting with 
$$ 
\frac{1}{2\pi i} \int_{2-i\infty}^{2+i\infty} -\frac{\zeta^{\prime}}{\zeta}(s) \frac{x^{s-1}}{s(s-1)}ds 
= \sum_{n\le x} \frac{\Lambda(n)}{n} \Big(1-\frac nx\Big). 
$$
Moving the line of integration to the left, and computing residues, we find that the 
above equals 
$$ 
\log x - (1+\gamma) + \frac{\log (2\pi)}{x} - \sum_{\rho} \frac{x^{\rho-1}}{\rho(\rho-1)} 
- \sum_{n=1}^{\infty} \frac{x^{-2n-1}}{2n(2n+1)}.
$$
The lemma follows.
\end{proof}

\begin{lemma}\label{propL1}   Let $q\ge 3$ and let $\chi \pmod q$
be a primitive Dirichlet character.  
 Suppose that GRH holds for $L(s,\chi)$. For any $x\geq 2$, there exists a real number $|\theta|\leq 1$ such that
\begin{align*}
\log|L(1,\chi)|&= \tRe\sum_{n\leq x}\frac{\chi(n)\Lambda(n)}{n\log n}
\frac{\log\left(\frac{x}{n}\right)}{\log x} +\frac{1}{\log x} \Big(\frac 12 \log \frac{q}{\pi} + 
\frac 12 \psi_0\Big(\frac{1+\frak a}{2}\Big)\Big) 
\\
&\hskip 1 in - \Big( \frac 1{\log x} + \frac{2\theta}{\sqrt{x}(\log x)^2}\Big)  |\tRe (B(\chi))| +\frac{2\theta}{x\log^2 x}.
\end{align*}
\end{lemma}
\begin{proof}  For any $\sigma \ge 1$, we consider 
$$
 \frac{1}{2\pi i}\int_{2-i\infty}^{2+i\infty}
-\frac{L'}{L}(\sigma+s,\chi)\frac{x^{s}}{s^2}ds=
\sum_{n\leq x}\frac{\chi(n)\Lambda(n)}{n^{\sigma}}\log(x/n).
$$
Shifting the contour to the left, we find that the integral also equals
$$
-\frac{L'}{L}(\sigma,\chi)\log x-\left(\frac{L'}{L}\right)'(\sigma,\chi)
-\sum_{\rho_\chi}\frac{x^{\rho_\chi-\sigma}}{(\rho_\chi-\sigma)^2}-
 \sum_{n=0}^{\infty}\frac{x^{-2n-{\frak a}-\sigma}}{(2n+\frak a+ \sigma)^2}.
$$
Therefore 
\begin{align*}
-\frac{L'}{L}(\sigma,\chi)= &\sum_{n\leq x}\frac{\chi(n)\Lambda(n)}{n^{\sigma}}
\frac{\log(x/n)}{\log x}+ \frac{1}{\log x}\left(\frac{L'}{L}\right)'(\sigma,\chi)\\
&\hskip 1 in +\frac{\theta x^{\frac 12-\sigma}}{\log x}\sum_{\rho_\chi} \frac{1}{|\rho_\chi|^2} 
+
\frac{\theta x^{-\sigma}}{\log x}\sum_{n=0}^{\infty}\frac{x^{-2n}}{(2n+1)^2}.
\end{align*}
Integrating both sides over $\sigma$ from $1$ to $\infty$ and taking real parts 
we conclude that 
\begin{align*}
\log |L(1,\chi)|&= \tRe\sum_{n\leq x}\frac{\chi(n)\Lambda(n)}{n\log n} 
\frac{\log(x/n)}{\log x}- \frac{1}{\log x}\tRe\frac{L'}{L}(1,\chi)\\
&\hskip 1 in + \frac{\theta}{\sqrt{x}(\log x)^2}
 \sum_{\rho_\chi}\frac{1}{|\rho_\chi|^2} + \frac{2\theta}{x (\log x)^2}.
 \end{align*}
The lemma follows upon noting that 
\begin{align*}
-\tRe \frac{L^{\prime}}{L}(1,\chi) 
&= -\tRe \frac{\xi^{\prime}}{\xi}(1,\chi) + \frac 12\log \frac{q}{\pi} +\frac 12\psi_0\Big(\frac{1+{\frak a}}{2}\Big) 
\\
&= - |\tRe (B(\chi))| + \frac 12 \log \frac{q}{\pi} + \frac 12 
\psi_0 \Big(\frac{1+\frak a}{2}\Big). 
\end{align*}
\end{proof}

 \begin{lemma}
 \label{lem6}  Assume RH.  For all $x\ge e$ we have 
 $$ 
 \sum_{n\le x} \frac{\Lambda(n)}{n\log n} \frac{\log (x/n)}{\log x} = 
 \log \log x + \gamma - 1 +\frac{\gamma}{\log x} +\frac{2B\theta}{\sqrt{x}(\log x)^2} 
 + \frac{\theta}{3x^3 (\log x)^2}. 
 $$ 
 \end{lemma}
 \begin{proof}  Analogously to the proof of Lemma \ref{propL1} we begin with, for any 
 $\sigma>1$, 
 $$ 
 \frac{1}{2\pi i} \int_{2-i\infty}^{2+i\infty} -\frac{\zeta^{\prime}}{\zeta}(s+\sigma)\frac{x^{s}}{s^2} ds 
 = \sum_{n\le x} \frac{\Lambda(n)}{n^{\sigma}} \log (x/n), 
 $$ 
 and moving the line of integration to the left, this equals 
 $$ 
 \frac{x^{1-\sigma}}{(\sigma-1)^2} - \frac{\zeta^{\prime}}{\zeta}(\sigma)\log x 
 - \Big(\frac{\zeta^{\prime}}{\zeta}\Big)^{\prime} (\sigma) - \sum_{\rho} \frac{x^{\rho-\sigma}}{(\rho-\sigma)^2} - \sum_{n=1}^{\infty} \frac{x^{-2n-\sigma}}{(2n+\sigma)^2}. 
 $$  
 Let $\sigma_0$ denote a parameter that will tend to $1$ from above, and integrate the two 
 expressions above for $\sigma$ from $\sigma_0$ to $\infty$.  Thus 
 \begin{align*}
 \sum_{n\le x }\frac{\Lambda(n)}{n^{\sigma_0} \log n} \log(x/n) 
& = \int_{\sigma_0}^{\infty} \frac{x^{1-\sigma}}{(\sigma-1)^2}d\sigma + (\log x) \log \zeta(\sigma_0) 
 + \frac{\zeta^{\prime}}{\zeta}(\sigma_0) \\
 &\hskip .5 in + \frac{2B\theta}{\sqrt{x} \log x} + 
 \frac{\theta}{3x^3 \log x}.
 \end{align*}
 Now, with $\alpha= (\sigma_0-1) \log x$, 
 $$ 
 \int_{\sigma_0}^{\infty} \frac{x^{1-\sigma}}{(\sigma-1)^2} d\sigma 
 = (\log x) \int_{\alpha}^{\infty} \frac{e^{-y}}{y^2} dy, 
 $$ 
 and upon integrating by parts we have 
 $$ 
 \int_\alpha^{\infty} \frac{e^{-y}}{y^2} dy = \frac{e^{-\alpha}}{\alpha} - \int_1^{\infty} 
 \frac{e^{-y}}{y} dy + \int_{\alpha}^{1} \frac{1-e^{-y}}{y} dy + \log \alpha.
 $$
 Now let $\sigma_0 \to 1^+$, so that $\alpha \to 0^+$, and 
 note that $\log \zeta(\sigma_0) = - \log (\sigma_0-1) + O(\sigma_0-1)$, 
$-\zeta^{\prime}/\zeta(\sigma_0) = 1/(\sigma_0-1) - \gamma +O(\sigma_0-1)$, 
and that $\int_0^{1} (1-e^{-y})dy/y -\int_1^{\infty} e^{-y} dy/y = \gamma$.  The 
lemma then follows.
 \end{proof}

\section{Proof of Theorem \ref{thmexplicit}}

\noindent  Unlike the other results of this paper, the following simple lemma is unconditional.

 \begin{lemma}\label{lem5}
Let $m\geq 3$ be an integer and $x\geq 2$ be a real number.  Then
$$\sum_{\substack{n\leq x\\(n,m)>1}}\Lambda(n)\log(x/n)
\leq \frac{1}{2}\omega(m)(\log x)^2,
$$
and
$$ \sum_{\substack{n\leq x\\(n,m)>1}}\frac{\Lambda(n)}{n}\left(1-\frac{n}{x}\right)\leq \sum_{\substack {n \\ (n,m)>1}} \frac{\Lambda(n)}{n} = \sum_{p|m}\frac{\log p}{p-1}.
$$
\end{lemma}
\begin{proof}
If $(n,m)>1$ and $n\leq x$ is a prime power, then $n=p^{\alpha}$ where $p|m$ and $\alpha\leq \log x/\log p$. Hence
\begin{align*}
\sum_{\substack{n\leq x\\(n,m)>1}}\Lambda(n)\log(x/n)&=
\sum_{\substack{p\leq x\\ p|m}}\sum_{\alpha\leq \log x/\log p}\log p\log\left(\frac{x}{p^{\alpha}}\right)\\
&= \sum_{\substack{p\leq x\\ p|m}}\log p\log x\left[\frac{\log x}{\log p}\right] -
\sum_{\substack{p\leq x\\ p|m}}\log^2 p\sum_{\alpha\leq \log x/\log p}\alpha\\
&\leq \frac{1}{2}\sum_{\substack{p\leq x\\ p|m}}\log p\log x\left[\frac{\log x}{\log p}\right] \leq \frac{1}{2}\omega(m)(\log x)^2.
\end{align*}
The second part of the lemma is evident. 
\end{proof}

Let $q\ge 3000$ and let $H$ be a proper subgroup of $G=({\Bbb Z}/q{\Bbb Z})^*$.  
Let $X$ be such that all primes $\ell \nmid q$ with $\ell \le X$ lie in the 
subgroup $H$.   
Let $\tilde H$ denote the group of Dirichlet characters $\chi \pmod q$ 
such that $\chi(n) =1$ for all $n \in H$.  Thus ${\tilde H}$ is a subgroup of 
the group of Dirichlet characters mod $q$, and note that $|{\tilde H}|$ equals 
$[G:H]$.  The assumption that all primes $\ell \le X$ with $\ell \nmid q$ 
lie in $H$ is equivalent to $\chi(\ell)=1$ for all $\chi \in {\tilde H}$.   Since $H$ is 
a proper subgroup, there exists a non-principal character $\chi \pmod q$ in ${\tilde H}$.  
Our strategy is to compare upper and lower bounds for $\tRe (S(X,\chi))$ where, 
as defined earlier, $S(X,\chi) = \sum_{n\le X} \Lambda(n) \chi(n) \log (X/n)$.  
Note that the character $\chi$ need not be primitive, and we let $\tilde \chi \pmod{\tilde q}$ 
denote the primitive character that induces $\chi$.   

 \subsection{Proof of Part 1}  Here we assume that 
 $X= (\log q+B(q))^2 \ge (\log q)^2 \ge 64$, and 
 seek a contradiction.  The assumption that $\chi(\ell)=1$ for all 
 primes $\ell$ below $X$ with $\ell\nmid q$  gives  
 $$ 
 S(X,\chi) = \sum_{n\le X} \Lambda(n) \log (X/n)  - \sum_{\substack{n\le X\\ (n,q)>1}} 
 \Lambda(n) \log (X/n).
 $$ 
 Using Lemma \ref{lem1} and Lemma \ref{lem5} (recall that $\chi(n) = 1$ for all $n\leq X$) we obtain the 
lower bound
\begin{equation} 
\label{lower2}
\tRe (S(X, \chi)) \geq X - \sqrt{X}\left(2|B|+2\omega(q)\frac{(\log\log q)^2}{\log q}\right) - \log (2\pi)\log X - 1 -2|B|,
\end{equation} 
where we used that $X\ge (\log q)^2$ and so $(\log X)^2/\sqrt{X} \le 4 (\log \log q)^2/\log q$.  

Next we work on the upper bound for $\tRe (S(X,\chi))$.  Note that 
\begin{equation} 
\label{eqn3.2} 
S(X,\chi) - S(X,\tilde\chi) = \theta \sum_{\substack {n\le X\\ (n,q/{\tilde q})>1} } 
\Lambda(n) \log(X/n) 
= \theta \frac{\omega(q/\tilde q)}{2} (\log X)^2. 
\end{equation}
Using Lemma \ref{lem2} for $\tilde \chi$,  and since ${\tilde E}_{\frak a}(X) \le -11/4$ for $X\ge 64$, 
we find 
\begin{equation} 
\label{eqn3.3} 
\tRe (S(X,\tilde \chi)) \le (2\sqrt{X}+2+\log X) |\tRe(B(\tilde \chi))| + 
\frac 12 \Big( \log \frac{\tilde q}{\pi} \Big) \log X - \frac{11}{4}. 
\end{equation}

We shall bound $|\tRe (B(\tilde \chi))|$ above using 
Lemma \ref{lem3}.  First note that by Lemmas \ref{lem4} and \ref{lem5}
\begin{eqnarray*}
\sum_{n\leq X} \frac{\Lambda(n){\tilde \chi}(n)}{n} \left(1-\frac{n}{X}\right)
&=& \sum_{n\leq X} \frac{\Lambda(n)}{n} \left(1-\frac{n}{X}\right) - \sum_{\substack{n\leq X\\(n, \tilde q) > 1}} \frac{\Lambda(n)}{n} \left(1-\frac{n}{X}\right)\\
&\geq&\max \Big(0,  \log (X) - \frac{8}{5}- \sum_{p|q} \frac{\log p}{p-1}\Big)\\ 
&\ge& \max\Big( 0, 2\log \log q -\frac{8}{5} -\sum_{p|q} \frac{\log p}{p-1}\Big) = A(q).\\
\end{eqnarray*} 
For $X \geq 64$, we have that $(1-\frac{1}{\sqrt{X}})^{-2} \leq 1 + \frac{5}{2\sqrt{X}}$ and $E_\mathfrak a (X) <-1/4$.    
Using this in Lemma \ref{lem3}, we get 
\begin{align*}
|\tRe(B(\tilde \chi))| &\le \Big( 1-\frac{1}{\sqrt{X}}\Big)^{-2} \frac 12 \log \frac{\tilde q}{\pi} - A(q) + E_{\frak a}(X) \le \frac{\log\tilde q}{2}+ \frac 25 -A(q).
\end{align*}
We use this bound in \eqref{eqn3.3} and combine that with \eqref{eqn3.2} to obtain 
an upper bound for $\tRe(S(X,\chi))$.  Since $\omega(q/{\tilde q}) \le (\log (q/\tilde q))/\log 2$,
 the resulting upper bound is largest when $\tilde q=q$.  
Thus 
\begin{align*}
\tRe (S(X,\chi)) &\le (2\sqrt{X}+2+ \log X) \Big(\frac{\log q}{2} + \frac 25 -A(q)\Big) + \frac{\log X}{2} \Big(\log 
 \frac q\pi\Big) - \frac {11}{4}  \\
 &\le \sqrt{X} \Big( \log q + \frac 45 -2 A(q) + \frac{\log X}{\sqrt{X}} \log q + \frac{\log q}{\sqrt{X}}\Big)-\frac{39}{20} \\ 
 &\le \sqrt{X} \Big( \log q + 2\log \log q + \frac 95 -2 A(q)\Big)-\frac{39}{20}.
\end{align*}
 
 Comparing the upper bound above with the lower bound \eqref{lower2} gives the 
 desired contradiction.

\subsection{Proof of Part 2}  Here we suppose that $X \geq (\log q)^2$ and 
seek a contradiction.  The proof follows the same lines as Part 1, with 
simplifications due to the assumption that $q$ has no prime factors below $X$, and 
a little more care with constants.   
Thus using Lemma \ref{lem1} 
we have the lower bound 
\begin{equation}\label{eqlow1}
\tRe S(X,\chi)= \sum_{n\le X} \Lambda(n) \log (X/n)
\ge X - 2|B|(\sqrt{X}+1) - \log (2\pi) \log X - 1.
\end{equation}

Now we turn to the upper bound.   Since $q$ has no prime factors below $X$ 
we have $S(X,\chi) = S(X,\tilde \chi)$.   
From Lemma \ref{lem3} and Lemma \ref{lem4} we have that 
$$ 
|\tRe (B(\tilde \chi))| \le \Big(1-\frac{1}{\sqrt{X}}\Big)^{-2} \Big( \frac 12 \log \frac {\tilde q}\pi - \log X + 
\frac{8}{5}
+ E_{\frak a}(X)\Big).  
$$ 
For $X\ge 64$ we have that $(1-1/\sqrt{X})^{-2} \le (1+5/(2\sqrt{X}))$, and we 
may check that for ${\frak a}=0$ or $1$, 
$$ 
-\frac 12 \log \pi -\log X +\frac 85 +E_{\frak a}(X) \le -\log X + \frac 76.
$$
Therefore, as $X\ge (\log q)^2$, and $\tilde q\le q$,
\begin{align*}
|\tRe (B(\tilde \chi))| &\le \Big( 1+ \frac{5}{2\sqrt{X}}\Big) \Big( \frac 12 \log \tilde q - \log X + \frac 76\Big) 
\\
&\le \frac {\log q}{2} +\frac 54 + \Big(1+\frac{5}{2\sqrt{X}}\Big) \Big(\frac 76- \log X\Big)\\
&\le \frac{\log q}{2} -\frac{\log X}{2} - \frac{4}{7}, 
\end{align*}
where the last bound follows upon using $X\ge 64$ together with a little calculus.
Using this in Lemma \ref{lem2}, and noting that $\tilde{E}_\af(x) < 0$ for $\frak a=0$ or $1$ and 
$x\ge 3$, we get  (recall $X\geq (\log q)^2$)
\begin{align} \label{equp1}
\tRe S(X,\chi)
&\le  (2+2\sqrt{X} + \log X) \Big( \frac 12 \log q -\frac 12\log X -\frac 47\Big) + \frac 12 \Big(\log \frac q\pi \Big) \log X \notag \\
&\le \sqrt{X}\log q + \log X\Big( - \frac 12\log (\pi X) -\frac {11}{7} \Big) 
-\frac {\sqrt{X}}7  - \frac 87 \notag\\
&\le \sqrt{X}\log q-\frac{\sqrt{X}}{7}-2\log X-\frac{8}{7}.
\end{align}
Comparing the bounds \eqref{eqlow1} and \eqref{equp1}  gives 
a contradiction, which proves the claimed result.  
 
 \section{The least prime in a given coset}  
  \noindent  As before, let ${\tilde H}$ denote the group of characters 
 $\chi \pmod q$ with $\chi(n)=1$ for all $n\in H$.  Recall that $|\tilde H|=
 [G:H]=h$, and given a coset $aH$ we have the orthogonality relation 
 \begin{equation}
 \label{cosorth}
 \frac 1h \sum_{\chi \in {\tilde H}} \overline{\chi(a)} \chi(n) 
 = \begin{cases}
 1 &\text{if  } n\in aH \\ 
 0 &\text{if  } n\notin aH.\\
 \end{cases}
 \end{equation}
Note that $q\geq 20000$.  Let $X$ be such that no prime below $X$ lies in the coset $aH$, and we assume below that $X\ge \max(10^9, ((h-1)\log q)^2)$.    

Set $S(X, \chi) = \sum_{n\le X} \Lambda(n) \chi(n) \log (X/n)$, 
 so that 
\begin{equation} 
\label{eqnnew1} 
\sum_{\substack{n\le X \\ n\in aH }} \Lambda(n) \log (X/n)  = \frac 1h \sum_{\chi \in \tilde H} 
\overline{\chi(a)} S(X,\chi).
\end{equation}
Our strategy is again to obtain upper and lower bounds for the quantity above, 
and then to derive a contradiction.

\subsection{Preliminary bounds}
 
First, consider the principal character $\chi_0 \pmod q$ which 
certainly belongs to the group $\tilde H$.  Since 
$$ 
 S(X,\chi_0) = \sum_{n\le X} \Lambda(n) \log (X/n) - 
 \sum_{\substack{n\le X \\ (n,q)>1} } \Lambda(n) \log (X/n), 
$$ 
using Lemmas \ref{lem1} and \ref{lem5}, we obtain that for $X\ge \max((h-1)^2\log^2 q, 10^9)$,
\begin{align}\label{eqnAPlem1}
|S(X, \chi_0) - X| &\leq 2|B|(\sqrt{X}+1) + (\log 2\pi)\log X + 1 + \frac{\omega(q)}{2} (\log X)^2\notag\\
&\le \frac{\sqrt{X}}{20} +\frac{\omega(q)}{2}(\log X)^2.
\end{align}  

For a non-principal character $\chi \in \tilde H$, 
let ${\tilde \chi} \pmod {\tilde q}$ denote the primitive character that induces $\chi$.  
By Lemma \ref{lem2} we find that 
\begin{equation}
\label{4.4} 
|S(X, \tilde \chi)| \le (2\sqrt{X}+2) |\tRe (B(\tilde \chi))| + \Big| \frac{\xi^{\prime}}{\xi}(0,
\tilde \chi)\Big| \log X + \Big| \frac12 \Big(\log \frac{\tilde q}{\pi}\Big) \log X + {\tilde E}_{\frak a}(X)\Big|.
\end{equation}
Now for $X\ge 10^9$  we know that $\tilde E_{\frak a}(X)<0$, and examining the 
 definition of ${\tilde E}_{\frak a}$ we find that 
\begin{equation} 
\label{4.41} 
\Big|\frac 12\Big(\log \frac{\tilde q}{\pi} \Big) \log X + {\tilde E}_{\frak a}(X)\Big| \le 
\frac 12 (\log X)\log \max\Big(\frac q{\pi},2X\Big).
\end{equation}
Recall from \eqref{eqn3.2} that 
\begin{equation} 
\label{4.5} 
|S(X,\chi)| \le |S(X,\tilde \chi)| + \omega(q/\tilde q) \frac{(\log X)^2}{2}. 
\end{equation}
We now invoke Lemma \ref{lem3}, taking there $x= 100$.  Since 
$|\xi^{\prime}/\xi(0,\tilde \chi)| = |\xi^{\prime}/\xi(0,\overline{\tilde \chi})| \ge |\tRe (B(\tilde \chi))|$ 
we obtain 
$$ 
\Big|\frac{\xi^{\prime}}{\xi}(0,\tilde \chi)\Big| \Big(1-\frac{2}{10} -\frac 1{100}\Big) 
\le \frac 12 \frac{99}{100} \log \frac{\tilde q}{\pi} + \sum_{n\le 100} \frac{\Lambda(n)}{n} \Big(1-\frac{n}{100}\Big).  
$$
It follows with a little computation that 
\begin{equation} 
\label{4.6} 
\Big| \frac{\xi^{\prime}}{\xi} (0,\tilde \chi) \Big| \le 
\frac{2}{3} \log \frac{  q}{\pi} + 4.
\end{equation}
Further, using Lemma \ref{lem3} (taking there $x=X$, and since $E_{\frak a}(X)\le -2/7$ for 
$X\ge 10^9$), we find that 
\begin{align*}
|\tRe (B(\tilde \chi))| 
&\le \Big(1-\frac{1}{\sqrt{X}}\Big)^{-2} \Big( \frac 12\Big(1-\frac 1X\Big) \log \frac{\tilde q}{\pi} - 
\tRe \sum_{n\le X} \frac{\Lambda(n)}{n} \tilde \chi(n) \Big(1-\frac nq\Big) -\frac 27 \Big). 
\end{align*} 
Further, 
\begin{align*}
\tRe \sum_{n\le X} \frac{\Lambda(n)}{n} 
\tilde \chi(n)\Big(1-\frac nX\Big) &\ge 
\tRe \sum_{n\le X} \frac{\Lambda(n)}{n} 
 \chi(n) \Big(1-\frac nX\Big) - \sum_{\substack{ n\le X \\ (n,q/\tilde q)>1}} 
\frac{\Lambda(n)}{n} \Big(1-\frac{n}{X}\Big)\\
&\ge \tRe \sum_{n\le X} \frac{\Lambda(n)}{n} 
 \chi(n) \Big(1-\frac nX\Big) - \Big(1-\frac 1X\Big) \sum_{p|(q/\tilde q)} \frac{\log p}{p-1}.
\end{align*} 
 Therefore, we find that for $X\ge 10^9$, the quantity $(2\sqrt{X}+2) |\tRe(B(\tilde \chi))|$ is 
 bounded by 
 \begin{align*} 
 \le \Big(\sqrt{X}+\frac{19}{6} \Big) 
 \Big( \log \frac{\tilde q}{\pi} +2 \sum_{p|(q/\tilde q)} \frac{\log p}{p-1} -2\tRe \sum_{n\le X} 
 \frac{\Lambda(n)}{n}\chi(n) \Big(1-\frac{n}{X}\Big) -\frac 47\Big). 
 \end{align*}
 Consider the quantity above together with the term $\omega(q/{\tilde q})(\log X)^2/2$ 
 appearing in \eqref{4.5}.  Since $\log \tilde q \le \log q - \sum_{p|(q/\tilde q)} \log p$, and 
 $X\ge 10^{9}$, we may check that the sum of these two quantities is largest when 
 $q/{\tilde q}$ is $6$.   Putting in this worst case bound, we find that 
 \begin{align*}
 (2\sqrt{X}+2)&|\tRe (B(\tilde \chi))| + \omega(q/\tilde q)\frac{(\log X)^2}{2} 
 \le (\log X)^2 \\
 &+ \Big( \sqrt{X}+\frac{19}{6} \Big) \Big(\log \frac{q}{\pi} +\log 2 -\frac 47 -2\tRe 
 \sum_{n\le X} \frac{\Lambda(n)}{n} \chi(n)\Big(1-\frac nX\Big) \Big).
 \end{align*}
 Combining this estimate with \eqref{4.4}, \eqref{4.41}, \eqref{4.5}, and \eqref{4.6} we conclude 
 that 
 \begin{align}
 \label{4.8}
  |S(X,\chi)| &\le \Big(\sqrt{X}+\frac{19}{6}\Big) \Big(\log q-\frac{51}{50} -2\tRe\sum_{n\le X} 
 \frac{\Lambda(n)}{n}\chi(n) \Big(1-\frac nX\Big) \Big) +(\log X)^2
\nonumber \\
 &+ \Big(\frac 23 \log \frac { q}\pi + 4\Big) (\log X) +  \frac{(\log X)}{2} \log \max\Big(\frac{q}{\pi},2X\Big).
  \end{align}

We sum the above over the $h-1$ non-principal characters of $\tilde H$.  
Note that, using the orthogonality relations and Lemma \ref{lem4}, 
$$ 
-\tRe \sum_{\substack{ \chi \in \tilde H \\ \chi\neq \chi_0} } \sum_{n\leq X}
\frac{\Lambda(n)}{n} \chi(n) \Big( 1-\frac nX \Big) 
\le  \sum_{n\le X} \frac{\Lambda(n)}{n} \Big(1-\frac nX\Big) \le \log X - \frac 32.  
$$ 
Thus, we conclude that 
 \begin{align*}
 \sum_{\substack{\chi\in \tilde H \\ \chi \neq \chi_0}} |S(X,\chi) | 
&\le \Big(\sqrt{X}+\frac{19}{6}\Big) \Big( (h-1)\log q - (h+2) +2\log X\Big) 
 +(h-1)(\log X)^2 \\
 &+(h-1)\log X  \Big(\frac 23\log \frac q\pi + 4+\frac12 \log \max\Big(\frac q\pi,2X\Big)\Big).
\end{align*}
Using that $q\ge 20000$ and  $X\ge \max(10^9, ((h-1) \log q)^2)$, we  may simplify the above bound, and 
obtain that 
\begin{align} 
\label{eqn4.9} 
\sum_{\substack{\chi\in \tilde H \\ \chi \neq \chi_0}} |S(X,\chi) | 
&\le \sqrt{X} \Big( (h-1)\log q - h +\frac{6}{5} + 3\log X\Big) + (h-1)(\log X)^2 \nonumber\\ 
&+(h-1) \frac{\log X}{2} \log \max \Big( \frac{q}{\pi},2X\Big). 
\end{align}
 Combining this with \eqref{eqnAPlem1} we obtain a lower bound 
 for the LHS of \eqref{eqnnew1}.



\subsection{Upper bound for (\ref{eqnnew1})}
  Now we turn to the upper bound.  By assumption there 
are no primes $p\le X$ with $p \in aH$.   Therefore  
\begin{equation*}
\sum_{\substack{n\le X\\ n\in aH}} \Lambda(n) \log (X/n) 
\le \sum_{\substack{p^{2k \le X} \\ p^{2k} \in aH }} (\log p) \log(X/p^{2k})  
+ \sum_{p^{2k+1} \le X } (\log p) \log (X/p^{2k+1}). 
\end{equation*} 
The second term above is, using Lemma \ref{lem1}, and a little computation using $X \ge 10^9$, 
\begin{align*}
\label{eqn4.6}
&\le \sum_{k\le \log X/(2\log 2)-1/2} (2k+1) \sum_{p\le X^{1/(2k+1)}} (\log p) \log (X^{1/(2k+1)}/p) 
\notag \\
&\le \sum_{k\le \log X/(2\log 2)-1/2} (2k+1)\Big (X^{1/(2k+1)} + \frac{X^{1/(4k+2)}}{20}\Big) 
\le \frac{\sqrt{X}}{7},
\end{align*}
so that
\begin{equation}
\label{eqn4.5}
\sum_{\substack{n\le X\\ n\in aH}} \Lambda(n) \log (X/n) 
\le \sum_{\substack{p^{2k} \le X\\ p^{2k} \in aH }} (\log p) \log(X/p^{2k}) + \frac{\sqrt{X}}{7}. 
\end{equation}

\subsection{The least prime in an arithmetic progression}

\noindent Using a computer we checked Corollary \ref{thmexplicit2} when $q\le 20000$.  
Suppose now that $q>20000$ and let $a\pmod q$ be an arithmetic progression with $(a,q)=1$ such that no prime below $X =(\phi(q) \log q)^2$ is $\equiv a\pmod q$.  

If $q$ has at least six distinct prime factors 
then $\phi(q) \ge \phi(2\cdot 3\cdot 5\cdot 7\cdot 11\cdot 13) = 5760$, 
while if $q$ has at most five prime factors then $\phi(q)\ge q \phi(2\cdot 3\cdot 5\cdot 7\cdot 11)/(2\cdot 3\cdot 5\cdot 7\cdot 11) 
>4155$.   In either case, $\phi(q) \ge 4156$.  By similar elementary arguments 
we may check that for $q>20000$ we have $2^{\omega(q)} \le q^{3/7}$, and that $\phi(q) \ge q^{5/6}$.  

Thus $X= (\phi(q)\log q)^2 \ge 10^9$, and we may use our work above; note that here $H$ is the 
trivial subgroup consisting of just the identity, and $h=\phi(q)$.  First we work out the lower bound for the quantity in
 (4.2).   Using (4.3) and (4.9), together with the bounds $2^{\omega(q)}\le q^{3/7}$ and $X\le (q\log q)^2$, we 
 find that 
 $$
\phi(q) \sum_{\substack{{n\le x}\\ {n\equiv a\pmod q}}} \Lambda(n) \log(X/n) 
\ge X -\sqrt{X} \Big( (\phi(q)-1)\log q  - \frac{24}{25} \phi(q) + \frac 32 + 3(\log X)\Big). 
$$ 

To obtain a corresponding upper bound, it remains to estimate the first term in (\ref{eqn4.5}).  
Note that the number of square roots of $a \pmod q$ is bounded by $2^{\omega(q)+1} \le 2q^{3/7}$.  Since $(\log p)(\log X/p^{2k}) \le (\log X)^2/8$ we find that 
\begin{align}
\label{eqn:APup}
\sum_{\substack{p^{2k} \le X \\ p^{2k} \in aH }} (\log p) \log(X/p^{2k})
&\leq \frac{(\log X)^2}{8} \sum_{\substack{n \le \sqrt{X} \\ n^2 \equiv a \textup{ (mod $q$)} }} 1  
\leq  q^{3/7} \frac{(\log X)^2}{4} \Big(\frac{\sqrt{X}}{q}+1\Big). \notag
\end{align}
Therefore 
$$ 
\sum_{\substack{{n\le x}\\ {n\equiv a\pmod q}}} \Lambda(n) \log(X/n) \le 
\frac{\sqrt{X}}{7} +  q^{3/7} \frac{(\log X)^2}{4} \Big(\frac{\sqrt{X}}{q}+1\Big).
$$ 
Using our bounds $\phi(q)\ge q^{5/6}$, $q >20000$ and $(q\log q)^2 \geq X\ge 10^9$, a little computation shows that 
the upper and lower bounds derived above give a contradiction. 




\subsection{The general case}
In general we bound the first term in (\ref{eqn4.5}) crudely by
\begin{equation*}  
\le 2 \sum_{n\leq \sqrt{X}} \Lambda(n)\log\frac{\sqrt{X}}{n} 
\le 2\Big(\sqrt{X} + \frac{X^{1/4}+1}{20}\Big), 
\end{equation*}
so that 
\begin{equation*}
\sum_{\substack {{n\le X } \\ {n\in aH} } } \Lambda(n) \log (X/n) 
\le \frac{11}{5} \sqrt{X}. 
\end{equation*}

On the other hand, using (4.3) and (4.9) together with the bounds $2^{\omega(q)} \le q^{3/7}$ and 
$X\ge \max(10^9, ((h-1)\log q)^2)$,  we obtain 
$$ 
h \sum_{\substack {{n\le X } \\ {n\in aH} } } \Lambda(n) \log (X/n) \ge X- \sqrt{X} \Big( (h-1)\log q - \frac{24}{25}h +\frac 54 
+ \frac 72 (\log X) + \frac{(\log X)^2}{3(h-1)}\Big). 
$$
Comparing this with our upper bound, we must have 
$$ 
\sqrt{X} \le (h-1)\log q + \frac 54 h + \frac 54 +\frac{7}{2} (\log X) + \frac{(\log X)^2}{3(h-1)}.
$$ 
Since $X\ge 10^9$, we have $\frac 72 \log X+ \frac{(\log X)^2}{3} \le \sqrt{X}/100$, so that from the 
above estimate we may first derive that $\sqrt{X} \le 2(h-1) \log q$.  Now inserting this bound into our 
estimate, we obtain the refined bound 
$$ 
\sqrt{X} \le (h-1) \log q + \frac 54(h+1) + 7 \log (2(h-1)\log q) + \frac{4 (\log (2(h-1)\log q))^2}{3(h-1)}. 
$$ 
The bound stated in Theorem \ref{thmcoset} follows from this with a little calculation. 


\section{Explicit bounds for $|L(1,\chi)|$: Proof of Theorem \ref{L1chi}}

\subsection{Upper bounds for $L(1,\chi)$}
Let $q\geq 10^{10}$ be a positive integer and $x\geq 100$ be a real number to be chosen later.   Lemma \ref{propL1} gives\begin{align*} 
\log |L(1,\chi)|&\le \tRe \sum_{n\le x} \frac{\chi(n)\Lambda(n)}{n\log n} \frac{\log (x/n)}{\log x} 
+ \frac{1}{2\log x} \Big(\log \frac q\pi + \psi_0\Big(\frac{1+\frak a}{2}\Big) \Big) 
\\ 
&- \Big(\frac{1}{\log x} -\frac{2}{\sqrt{x}(\log x)^2}\Big) |\tRe (B(\chi))| + \frac{2}{x(\log x)^2}. 
\end{align*}
Invoking Lemma \ref{lem3} (with the same value of $x$ above), we have 
$$ 
|\tRe (B(\chi))| \ge \Big(1+\frac {1}{\sqrt{x}}\Big)^{-2} \Big( \frac 12 \Big(1-\frac 1x\Big) 
\log \frac{q}{\pi} -\tRe  \sum_{n\le x} \frac{\Lambda(n)\chi(n)}{n} \Big(1-\frac nx\Big) + E_{\frak a}(x)\Big). 
$$ 
Now for $x\ge 100$, 
$$ 
-E_{\frak a}(x) \Big(1+\frac1{\sqrt{x}}\Big)^{-2} \Big(\frac{1}{\log x} -\frac{2}{\sqrt{x}(\log x)^2} 
\Big) + \frac{1}{2\log x} \psi_0 \Big(\frac{1+\frak a}{2}\Big) + \frac{2}{x(\log x)^2} \le 0, 
$$ 
and therefore 
\begin{align*} 
\log |L(1,\chi)|
 &\le \tRe \sum_{n\le x} \frac{\chi(n)\Lambda(n)}{n\log n} \frac{\log (x/n)}{\log x} 
+ \frac{\log (q/\pi)}{(\sqrt{x}+1)\log x} \Big(1+\frac{1}{\log x} \Big) \\
&+ \Big(\frac{1}{\log x} -\frac{2}{\sqrt{x}(\log x)^2} \Big) 
\Big(1+\frac 1{\sqrt{x}}\Big)^{-2} \tRe \sum_{n\le x} \frac{\Lambda(n)\chi(n)}{n} 
\Big(1-\frac{n}{x}\Big).
\end{align*}
The right hand side above is largest when $\chi(p)=1$ for all $p\le x$, and so 
\begin{equation*} 
\label{5.1} 
\log |L(1,\chi)| \le \sum_{n\le x} \frac{\Lambda(n)}{n\log n} \frac{\log (x/n)}{\log x} 
+\frac{1}{\log x} \sum_{n\le x} \frac{\Lambda(n)}{n} \Big(1-\frac{n}{x}\Big) 
+ \frac{\log q}{(\sqrt{x}+1)\log x}\Big(1+\frac{1}{\log x}\Big). 
\end{equation*}
Appealing now to Lemmas \ref{lem4} and \ref{lem6} we conclude that 
$$ 
\log |L(1,\chi)| \le \log \log x +\gamma -\frac{1}{\log x} + \frac{\log q}{\sqrt{x} \log x}\Big(1+\frac{1}{\log x}\Big).
$$ 
Choosing $x= (\log q)^2/4$, so that $x\ge 130$ for $q \ge 10^{10}$, we deduce that 
$$ 
|L(1,\chi)| \le e^{\gamma} \Big(\log x +1 + \frac{3.1}{\log x} \Big) 
\le 2e^{\gamma} \Big( \log \log q -\log 2 +\frac 12 + \frac{1}{\log \log q}\Big).
$$ 
This proves the stated upper bound for $|L(1,\chi)|$.

\subsection{Lower bounds for $|L(1,\chi)|$}  
The argument proceeds similarly to the one for upper bounds.  Let $q\geq 10^{10}$ be a positive integer. As before  choose $x=(\log q)^2/4$ so  that $x\geq 132$.
  Lemma \ref{propL1} gives
  \begin{align*}
\log |L(1,\chi)|& \ge \tRe \sum_{n\le x} \frac{\Lambda(n)\chi(n)}{n\log n} \frac{\log (x/n)}{\log x} 
+ \frac 1{2\log x} \Big( \log \frac{q}{\pi } +\psi_0 \Big(\frac{1+\frak a}{2}\Big)\Big) \\
&- \Big( \frac{1}{\log x} + \frac{2}{\sqrt{x}(\log x)^2}\Big) | \tRe (B(\chi))| - \frac{2}{x(\log x)^2}. 
\end{align*}
From Lemma \ref{lem3} (with the same value of $x$ above) we find 
that 
$$ 
|\tRe (B(\chi))| \le \Big(1-\frac{1}{\sqrt{x}}\Big)^{-2} \Big( \frac 12\Big(1-\frac 1x\Big) \log\frac{q}{\pi} 
- \tRe \sum_{n\le x} \frac{\Lambda(n)\chi(n)}{n} \Big(1-\frac nx\Big) + E_{\frak a}(x)\Big). 
$$ 
Now for $x\ge 100$ 
$$ 
-E_{\frak a}(x) \Big(1-\frac 1{\sqrt{x}}\Big)^{-2} 
\Big(\frac{1}{\log x}+\frac{2}{\sqrt{x}(\log x)^2} \Big) + \frac{1}{2\log x}\psi_0\Big(\frac{1+\frak a}{2}\Big) -\frac{2}{x(\log x)^2} \ge 0, 
$$ 
and therefore  
\begin{align}\label{lowerL1}
 \log |L(1,\chi)| &\ge \tRe \sum_{n\le x} 
\frac{\Lambda(n)\chi(n)}{n\log n} \frac{\log (x/n)}{\log x} 
- \frac{\log (q/\pi)}{(\sqrt{x}-1)\log x} \Big(1+\frac{1+1/\sqrt{x}}{\log x}\Big) \nonumber \\
&+ \Big(1-\frac 1{\sqrt{x}}\Big)^{-2} \Big(\frac 1{\log x}+\frac{2}{\sqrt{x}(\log x)^2} 
\Big) \tRe \sum_{n\le x}\frac{\Lambda(n)\chi(n)}{n} \Big(1-\frac nx\Big). 
\end{align}
 
From Lemma \ref{lem4} it follows that for $x\ge 132$,
$$
\tRe\sum_{n\le x}\frac{\Lambda(n)\chi(n)}{n} \Big(1-\frac nx\Big)\geq -\log x+1,
$$
and therefore, with a little calculation, 
$$ 
\Big( \Big( 1-\frac{1}{\sqrt{x}}\Big)^{-2} \Big(\frac{1}{\log x}+\frac{2}{\sqrt{x}(\log x)^2}\Big) - 
\frac{1}{\log x} \Big) \tRe \sum_{n\le x} \frac{\Lambda(n)\chi(n)}{n} \Big(1-\frac nx\Big) 
\ge - \frac{2}{\sqrt{x}}.
$$
Using this bound in \eqref{lowerL1}  we obtain 
\begin{align}
\label{eqn5.2}
\log |L(1,\chi)| &\ge \tRe\sum_{n\le x} \Lambda(n) \chi(n) 
\Big(\frac{1}{n\log n} -\frac{1}{x\log x} \Big) \nonumber \\
&- \frac{\log (q/\pi)}{(\sqrt{x}-1)\log x} 
\Big(1+\frac{1+1/\sqrt{x}}{\log x}\Big) -\frac{2}{\sqrt{x}}. 
\end{align}
The next lemma shows that the sum over $n$ above is smallest 
when $\chi(p)=-1$ for all $p\le x$.  

\begin{lemma}
\label{lemma5.1} 
For $x\ge 100$ we have 
$$ 
\tRe \sum_{n\le x} \Lambda(n) \chi(n) 
\Big( \frac{1}{n\log n} -\frac{1}{x\log x}\Big) 
\ge \sum_{p^k \le x} \Lambda(p^k) (-1)^k \Big( \frac{1}{p^k \log p^k} -\frac{1}{x\log x}\Big). 
$$ 
\end{lemma} 
\begin{proof}  Consider the contribution of the powers of a single prime $p\le x$ to both sides of the 
inequality above; we claim that the contribution to the left hand side is at least as large as the contribution to the 
right hand side.  If $\chi(p) =0$ then the contribution to the left hand side is zero, and the contribution to 
the right hand side is negative.  If $\chi(p)\neq 0$ then writing $\chi(p) = -e(\theta)$, we see that the difference 
in the contributions to the left hand side and right hand side equals 
$$ 
(\log  p) \sum_{k\le \log x/\log p} (-1)^{k-1} (1-\cos (k\theta)) \Big( \frac{1}{p^k \log p^k } - \frac 1{x\log x}\Big). 
$$  
If $p\ge 3$ then using $(1-\cos(k\theta)) \le k^2 (1-\cos \theta)$ we see that the quantity above 
is 
$$ 
\ge (\log p) (1-\cos \theta) \Big( \frac{1}{p\log p} -\frac{1}{x\log x} - \sum_{j=1}^{\infty} \frac{(2j)^2}{p^{2j} \log p^{2j}}\Big) 
\ge 0. 
$$
For $p=2$ we use the bound $0\le (1-\cos (k\theta)) \le k^2(1-\cos \theta)$ for $k\ge 6$ and compute explicitly  
the trigonometric polynomial arising from the first five terms.  With a little computer calculation the lemma follows in this case. 
\end{proof}

Note that 
\begin{align}
\label{eqn5.3}
\sum_{p^k \le x}  \Lambda(p^k) (-1)^k& \Big( \frac{1}{p^k \log p^k} -\frac{1}{x\log x}\Big) 
= -\sum_{n\le x} \frac{\Lambda(n)}{n \log n} \frac{\log (x/n)}{\log x} \nonumber \\
&-\frac{1}{\log x} 
\sum_{n\le x} \frac{\Lambda(n)}{n} \Big( 1-\frac nx\Big) + 2 \sum_{m^2\le x} 
\Lambda(m) \Big(\frac{1}{m^2 \log m^2} -\frac{1}{x\log x}\Big). 
\end{align}
The third term above is
$$ 
\log \zeta(2) - \sum_{m^2>x} \frac{\Lambda(m)}{m^2\log m} -2 \sum_{m^2\le x} \frac{\Lambda(m)}{x\log x} 
\ge \log \zeta(2) - \frac{3}{2\sqrt{x}},   
$$
where we bound the sums over prime powers trivially by replacing them with sums over odd numbers and powers of $2$.  
The first two terms in \eqref{eqn5.3} are handled using Lemmas \ref{lem4} and \ref{lem6}.   Putting these 
together, and using $x\ge 132$, we conclude that 
$$ 
\sum_{p^k \le x}  \Lambda(p^k) (-1)^k \Big( \frac{1}{p^k \log p^k} -\frac{1}{x\log x}\Big) 
\ge - \log \log x - \gamma  + \log \zeta(2)  +\frac{1}{\log x} - \frac{8}{5\sqrt{x}}.
$$ 
Inserting this estimate and Lemma \ref{lemma5.1} into \eqref{eqn5.2}, and noting that $\log (q/\pi) \le 2\sqrt{x}-1$, 
we obtain with a little calculation 
$$ 
\log |L(1,\chi)| \ge - \log \log x - \gamma +\log \zeta(2) -\frac{1}{\log x} -\frac{2}{(\log x)^2} - \frac{4}{\sqrt{x}}.
$$ 
Exponentiating this, and using $x\ge 132$, we obtain the stated lower bound.

\section{Asymptotic bounds} 

\noindent  Let  $\delta$ be a fixed positive real number, and 
let $K(s)$ denote a function holomorphic in a 
region containing $-1/2 - \delta <  \text{Re}(s) \le \frac 12+\delta$, save for possibly a simple pole at $s = -1/2$ with residue $r$.
 Further suppose that for all $s$ in this region bounded away from $-1/2$ we have 
$|K(s)|\ll 1/(1+|s|^2)$. 
For $\xi >0$ define the inverse Mellin transform 
$$ 
{\widetilde K}(\xi) = \frac{1}{2\pi i} \int_{c-i\infty}^{c+i\infty} K(s) \xi^{-s} ds,
$$ 
where the integral is over any vertical line with $-\frac 12 < c\leq \frac 12 + \delta$.  If $\xi \ge 1$ then by taking the $c=1/2$ in 
the integral above, and using $|K(s)|\ll 1/(1+|s|^2)$, we find that ${\widetilde K}(\xi) \ll 1/\sqrt{\xi}$.  If $\xi\le 1$ then moving the line of integration to $c=-1/2-\delta/2$ (encountering potentially a pole at $-1/2$) we find that ${\widetilde K}(\xi) \ll \sqrt{\xi}$.    Thus 
\begin{equation} 
\label{6.1}
|{\widetilde K}(\xi)| \ll \min( \xi,1/\xi)^{\frac 12}.
\end{equation} 
Finally we assume that $K$ is such that ${\widetilde K}(\xi) >0$ for all $\xi >0$.

As before, let $G = (\mathbb{Z}/q\mathbb{Z})^*$ and $H$ be a proper subgroup of $G$.  Let $X$ be such that all primes $\ell \nmid q$ with $\ell \le X$ lie in the 
subgroup $H$.  We may assume that $X \gg (\log q)^2$.   Further, let ${\tilde H}$ denote the group of characters 
 $\chi \pmod q$ with $\chi(n)=1$ for all $n\in H$.   Recall that $|\tilde H|=
 [G:H]=h$, and that  the orthogonality relation \eqref{cosorth} holds.  
 
 \begin{lemma} 
 \label{lemma6.1}  Assume GRH.  
 Let $\chi \pmod q$ be in ${\tilde H}$.   If $\chi$ is the principal character 
 then 
  $$ 
 \sum_{n} \frac{\Lambda(n) \chi(n)}{\sqrt{n}} {\tilde K}(n/x) = K(1/2){\sqrt{x}} + 
 O\Big(1+ \log q\frac{\log x}{\sqrt{x}} \Big). 
 $$ 
 If $\chi$ is non-principal then  
 $$ 
\textup{Re} \sum_{n} \frac{\Lambda(n)\chi(n)}{\sqrt{n}} {\tilde K}(n/x) = \theta (1+o(1)) 
\frac{\log q}{2\pi } \int_{-\infty}^{\infty} |K(it)|dt  + O\Big(1+ \log q \frac{\log x}{\sqrt{x}}\Big). 
 $$ 
 \end{lemma} 
 \begin{proof}   Consider first the case when $\chi$ is non-principal.  Let ${\tilde \chi} \pmod{\tilde q}$ denote 
 the primitive character that induces $\chi$ and let 
 $\xi(s,\tilde{\chi})$ denote the corresponding completed $L$-function. We 
 start with  
 \begin{equation} 
 \label{6.2} 
I= \frac{1}{2\pi i} 
 \int_{(1/2+\delta)} -\frac{\xi^{\prime}}{\xi}(s+\tfrac 12,\tilde \chi) K(s) x^s  ds,
 \end{equation}
 which we evaluate, as usual, in two ways.  Note that 
 $$ 
 \frac{1}{2\pi i} \int_{(1/2+\delta)} \Big( -\frac 12 \log \frac{\tilde q}{\pi}  \Big) K(s) x^s ds   \ll \frac{\log q}{\sqrt{x}},
 $$
 upon shifting the integral to $\tRe s = -1/2 -\delta$ and possibly picking up a pole at $s=-1/2$.
Further, by moving the line of integration to Re$(s)=0$, we find that 
$$ 
 \frac{1}{2\pi i} \int_{(1/2+\delta)} -\frac 12 \frac{\Gamma^{\prime}}{\Gamma}\Big(
 \frac{1/2+s+{\frak a}}{2}\Big) K(s) x^s ds \ll \int_{-\infty}^{\infty} \log (2+|t|) |K(it)| dt \ll 1.
 $$
Now by expanding $-L^{\prime}/L$ into its Dirichlet series we have 
 $$ 
 \frac{1}{2\pi i} \int_{(1/2+\delta)} -\frac{L'}{L}(\tfrac 12+s,\tilde\chi) K(s) x^s ds =  \sum_{n} \frac{\Lambda(n){\tilde \chi}(n)}{\sqrt{n}} {\tilde K}(n/x).
 $$ 
 Note that 
 \begin{align*}
 \sum_{n} \frac{\Lambda(n){\tilde \chi}(n)}{\sqrt{n}} {\tilde K}(n/x) 
- \sum_{n} \frac{\Lambda(n)\chi(n)}{\sqrt{n}} {\tilde K}(n/x)
 \ll \sum_{(n,q)>1} \frac{\Lambda(n)}{\sqrt{n}} |{\tilde K}(n/x)|, 
\end{align*}
 and using \eqref{6.1} this is 
 $$ 
 \ll \sum_{p|q} (\log p) \sum_{k=1}^{\infty} \frac{1}{p^{k/2}} \min \Big( \frac{\sqrt{x}}{p^{k/2}}, 
\frac{p^{k/2}}{\sqrt{x}}\Big) \ll \sum_{p|q} (\log p) \Big(\frac{1}{\sqrt{x}} + \frac{\log x}{\sqrt{x}\log p}\Big)\ll (\log q) \frac{\log x}{\sqrt{x}}. 
 $$
 From these observations, we conclude that  
 \begin{equation} 
 \label{6.3} 
 I=\sum_{n} \frac{\Lambda(n)\chi(n)}{\sqrt{n}} {\tilde K}(n/x) + O\Big(1+\log q \frac{\log x}{\sqrt{x}}\Big). 
 \end{equation} 
 
 Now evaluate the integral in \eqref{6.2} by shifting the line of integration to Re$(s)=- 1/2-\delta/2$.  Thus, with $\gamma$ running over the ordinates of zeros of $\xi(s,\tilde \chi)$,
 $$ 
 I = -\sum_{\gamma} K(i\gamma) x^{i\gamma}  -r\frac{\xi^{\prime}}{\xi}(0,\tilde \chi)x^{-1/2} + \frac{1}{2\pi i} \int_{(-1/2-\delta/2)} 
  -\frac{\xi^{\prime}}{\xi}(\tfrac 12+s,\tilde \chi)K(s) x^s ds. 
 $$
 Using now the functional equation for $\xi$, we 
 find that the integral on the right hand side is bounded by $\ll x^{-1/2-\delta/2}$.  
 
  Recall that $\tRe \frac{\xi^{\prime}}{\xi}(0,\tilde \chi) \ll \log q$, so that
  $$ 
  \text{Re}(I) = \theta \sum_{\gamma} |K(i\gamma)| + O\bfrac{\log q}{\sqrt{x}} = \theta (1+o(1)) \frac{\log {\tilde q}}{2\pi} 
  \int_{-\infty}^{\infty} |K(it)|dt + O\bfrac{\log q}{\sqrt{x}},  
  $$
 where the final estimate follows from an application of the explicit formula (see Theorem 5.12 of \cite{IK}).  
 This establishes our lemma for non-principal characters.  For principal characters, we have that
  \begin{align*}
  \sum_{n} \frac{\Lambda(n) }{\sqrt{n}} {\tilde K}(n/x) 
  &= \frac{1}{2\pi i}\int_{(1/2 + \delta/2)}-\frac{\zeta'}{\zeta}(1/2+s)K(s)x^s ds\\
  &=K(1/2)\sqrt{x} - \frac{1}{2\pi i}\int_{(-1/2-\delta/2)}\frac{\zeta'}{\zeta}(1/2+s)K(s)x^s ds - \sum_{\gamma} K(i\gamma) x^{i\gamma}\\
  &= K(1/2) \sqrt{x} + O(1).
  \end{align*}
  On the other hand, similar to before
 \begin{align*}
 \sum_{n} \frac{\Lambda(n){\chi_0}(n)}{\sqrt{n}} {\tilde K}(n/x) 
- \sum_{n} \frac{\Lambda(n)}{\sqrt{n}} {\tilde K}(n/x)
 \ll (\log q) \frac{\log x}{\sqrt{x}}. 
 \end{align*}
 From these observations, we conclude that  
$$
 I=\sum_{n} \frac{\Lambda(n)\chi(n)}{\sqrt{n}} {\tilde K}(n/x) + O\Big(1+\log q \frac{\log x}{\sqrt{x}}\Big). 
$$
 \end{proof}

\begin{prop}\label{prop:theoretical}  Assume GRH.  Keep the notations above, and recall that $X$ is such that 
all primes $\ell \nmid q$ with $\ell \le X$ lie in the subgroup $H$.  Then, for any $\lambda>0$ 
we have 
$$ 
\Big(h \int_0^{\lambda} {\tilde K}(u) 
\frac{du}{\sqrt{u}} - K(1/2)\Big) 
\sqrt{X} \le (1+o(1)) \sqrt{\lambda} (h-1) (\log q) \Big(\frac{1}{2\pi} \int_{-\infty}^{\infty} |K(it)|dt \Big). 
$$  
\end{prop}

\begin{proof}  We may assume that $X\gg (\log q)^2$.  We shall use Lemma \ref{lemma6.1} for each character $\chi \in {\tilde H}$, 
and taking $x=X/\lambda$.   Thus we obtain that 
$$ 
\sum_{\chi \in{\tilde H}} \text{Re} \sum_{n}\frac{ \Lambda(n)\chi(n)}{\sqrt{n}} {\tilde K}(n/x) 
\le K(1/2)\sqrt{x} +  (1+o(1))(h-1) \log q \Big(\frac{1}{2\pi} \int_{-\infty}^{\infty} |K(it)|dt\Big).
$$
From \eqref{cosorth}, and since ${\tilde K}(n/x)\ge 0$ for all $n$, we see that the left hand side above is 
$$ 
\ge h \sum_{n\le X} \frac{\Lambda(n)}{\sqrt{n}} {\tilde K}(n/x) +O\Big(h\sum_{(n,q)>1} \frac{\Lambda(n)}{\sqrt{n}} {\tilde K}(n/x)\Big) 
= h \int_{1}^{X}{\tilde K}(t/x) \frac{dt}{\sqrt{t}} + o(h\log q). 
$$ 
After a change of variables $u=t/x$, the 
above becomes 
$$ 
h\sqrt{x} \int_0^{\lambda} {\tilde K}(u)\frac{du}{\sqrt{u}} + o(h\log q). 
$$ 
The proposition follows with a little rearrangement.  
 \end{proof}

\subsection{Limitations of the method}  
 
By Mellin inversion we have that 
$K(1/2) = \int_0^{\infty} {\tilde K}(u) du/\sqrt{u}$, and note also that 
${\tilde K}(u) \le \frac{1}{2\pi} \int_{-\infty}^{\infty} |K(it)|dt$.  Therefore 
$$ 
h\int_0^{\lambda} {\tilde K}(u)\frac{du}{\sqrt{u}} - 
{K(1/2)}\le (h-1) \int_0^{\lambda} {\tilde K}(u)\frac{du}{\sqrt{u}} 
\le (2h-2) \sqrt{\lambda} \frac{1}{2\pi} \int_{-\infty}^{\infty} |K(it)| dt.  
$$ 
Hence Proposition \ref{prop:theoretical} cannot lead to a 
bound for $X$ that is better than $(\frac{1}4+o(1))(\log q)^2$.

\subsection{Bounds for  large $h$}  Take $K(s) =(\frac{ (e^{\alpha s}-e^{-\alpha s})}{s})^2$ and $\lambda=1$.  
Note that the inverse Mellin transform for $K$ is  ${\tilde K}(u) = \max(0, 2\alpha -|\log u|)$.     
Now 
$$ 
\frac{1}{2\pi} \int_{-\infty}^{\infty} |K(it)|dt = \frac{1}{2\pi} \int_{-\infty}^{\infty} \Big(\frac{2\sin(\alpha t)}{t}\Big)^2 dt =  2\alpha,
$$ 
and 
$$ 
\int_0^1 {\tilde K}(u) \frac{du}{\sqrt{u}} = \int_{e^{-2\alpha}}^1 \frac{2\alpha+\log u}{\sqrt{u}} du 
= 4\alpha-4+4e^{-\alpha}. 
$$ 
Thus Proposition \ref{prop:theoretical} implies that 
$$ 
\Big( h(4\alpha-4+4e^{-\alpha}) - 4(e^{\alpha/2}-e^{-\alpha/2})^2\Big) 
\sqrt{X} \le (1+o(1)) (2\alpha (h-1)) (\log q). 
$$ 
Selecting $\alpha = \frac 12 \log (2h)$, it follows that for $h\ge 4$ 
$$ 
X\le \Big(\frac 14+o(1)\Big)  \Big(1-\frac 1h\Big)^2 \Big(\frac{\log (2h)}{\log (2h)-4}\Big)^2(\log q)^2.
$$
For large $h$, the above bound is about $(1/4+o(1)) (\log q)^2$ which is 
of the same quality as the limit of the method, but the convergence to this limit is 
quite slow.  

\subsection{Bounds for small $h$} For smaller values of $h$, 
better estimates may be obtained using the kernel $K(s) = \Gamma(s+1/2)$.  Note that $K(s)$ satisfies the conditions stated at the beginning of this section and that ${\tilde K}(u) = \sqrt{u} e^{-u} \geq 0$ for all $u$.  Noting that $K(1/2) = \Gamma(1) = 1$, we apply Proposition \ref{prop:theoretical} to see that
$$\left( h \int_0^\lambda e^{-u} du - 1\right) \sqrt{X} \le \sqrt{\lambda} (h-1) \frac{\log q}{2\pi} \int_{-\infty}^\infty \left|\Gamma(1/2 + it) \right| dt, 
$$so that
\begin{equation*}
\sqrt{X} \le \frac{\sqrt{\lambda}(h-1) \log q}{2\pi (h-1-he^{-\lambda})} \int_{-\infty}^\infty \left|\Gamma(1/2 + it) \right| dt.
\end{equation*}

When $h=2$, we choose $\lambda = 2.452$, which is more or less optimal, and 
find that $X < (0.794+o(1))(\log q)^2$.  For $h=3$, we choose $\lambda = 2.025$, and find that $X < (0.7+o(1))(\log q)^2$.  For $h = 4$, we choose $\lambda= 1.825$ and get that $X < (0.66+o(1))(\log q)^2$.  We may apply this for other smaller values of $h$ and get progressively better bounds as $h$ increases, with $0.545(\log q)^2$ being the limit for this test function.

 \section*{Aknowledgements}
We thank Emanuel Carneiro and Micah Milinovich for drawing our attention to an error in Lemma 6.1 of the previous version of the paper, which affects the asymptotic bounds in Theorems 1.2 and 1.3 there. These results are corrected in this updated version.

\end{document}